\documentclass[10pt,draft,reqno]{amsart}
     \makeatletter
     \def\section{\@startsection{section}{1}%
     \z@{.7\linespacing\@plus\linespacing}{.5\linespacing}%
     {\bfseries
     \centering
     }}
     \def\@secnumfont{\bfseries}
     \makeatother
\newtheorem{theorem}{Theorem}[section]
\newtheorem{lemma}[theorem]{Lemma}
\newtheorem{proposition}[theorem]{Proposition}
\newtheorem{corollary}[theorem]{Corollary}
\theoremstyle{definition}
\newtheorem{definition}[theorem]{Definition}

\theoremstyle{remark}
\newtheorem{remark}[theorem]{Remark}
\numberwithin{equation}{section}
\usepackage{amsmath,amssymb,upref,amsthm,eucal}
\usepackage{bm}
\usepackage{graphicx}
\usepackage{graphics}
\usepackage[swedish,english]{babel}
\newcommand\Atop[2]{\genfrac{}{}{0pt}{}{#1}{#2}}
\newcommand{\Ab}{\mathbf{A}}
\newcommand{\AF}{\mathfrak{A}}
\newcommand{\AK}{{\bf\mathcal{A}}}
\newcommand{\EE}[1]{I\!\!E\left[#1\right]}
\newcommand{\EXP}{I\!\!E}
\newcommand{\fracN}{{\tfrac{1}{N}}}
\newcommand{\MM}{\mathbb{M}}
\newcommand{\NN}{I\!\!N}
\newcommand{\norm}[1]{\left\vert#1\right\vert}
\newcommand{\Norm}[1]{\left\Vert#1\right\Vert}
\newcommand{\PP}[1]{I\!\!P\left[#1\right]}
\newcommand{\RR}{I\!\!R}
\newcommand{\XX}{\mathbb{X}}

\newcommand{\Pmass}{\mathbf{P}}
\newcommand{\mub}{\bm{\mu}}
\newcommand{\rhb}{\bm{\rho}}

\newcommand{\ub}{\bm{u}}
\newcommand{\xb}{\bm{x}}
\newcommand{\yb}{\bm{y}}

\newcommand{\prtl}{\partial}

\begin{document}

\title[]{An Epsilon-Nash Equilibrium for Non-linear Markov Games of Mean-field-type on Finite Spaces}

\author{Rani Basna}
\address{Rani Basna: Department of Mathematics, Linnaeus University, V\"axj\"o, 351 95, Sweden}
\email{rani.basna@lnu.se}

\author{Astrid Hilbert}
\address{Astrid Hilbert: Department of Mathematics, Linnaeus University, V\"axj\"o, 351 95, Sweden}
\email{astrid.hilbert@lnu.se}

\author[Vassili N. Kolokoltsov]{Vassili N. Kolokoltsov}
\address{Vassili N. Kolokoltsov: Department of Statistics, University of Warwick, Coventry, CV4 7AL, UK}
\email{v.kolokoltsov@warwick.ac.uk}



\begin{abstract}
We investigate mean field games from the point of view of a large number of indistinguishable players which eventually converges to infinity. The players are weakly coupled via their empirical measure. The dynamics of the individual players is governed by pure jump type propagators over a finite space. Investigations are conducted in the framework of non-linear Markov processes. We show that the individual optimal strategy results from a consistent coupling of an optimal control problem with a forward non-autonomous dynamics. In the limit as the number $N$ of players
goes to infinity this leads to a jump-type analog of the  well-known non-linear McKean-Vlasov dynamics. The case where one player has an individual preference different from the ones of the remaining players is also covered. The two results combined reveal a
{\small $\fracN$}-Nash Equilibrium for the approximating system
of $N$ players.
\end{abstract}

\maketitle

\parindent = 0pt
\section{Introduction}

Mean field game theory is a type of dynamic Game theory where the agents are coupled with each other by their individual dynamics and their empirical mean. The objective of each agent, given in terms of the so called cost function, does not only depend on her own preference and decision but also on the decisions of the other players. All in all it is a mathematical tool to describe a control problem with a large number $N$ of agents where the impact of the individual decisions of the other agents is becoming extremely weak compared to the overall impact as $N$ increases to infinity.
The limiting model emerges from the fact that each agent constructs her
strategy from her own state and from the state of the empirical mean of an
infinite number of co-agents of hers and results in a decoupled dynamics and
objective which depend on the law of her dynamics.

\vskip 0.6cm
The mean field approach has been independently developed by J.-M. Lasry and P.-L.
Lions in a series of papers see \cite{OLL} and the references therein using
nonlinear PDE's and by M. Huang, P. Caines, Malham\'{e}, see \cite{HCM1}
\cite{HCM2} in the setting of stochastic processes, see also \cite{CA}.
\vskip 0.2cm
The investigations in this work are carried out in the framework of non-linear Markovian propagators, respectively time inhomogeneous nonlinear Feller processes, which was developed by Vassili Kolokoltsov \cite{K3} \cite{K4}. We focus on propagators related to processes of pure jump type with finite intensity measure on a finite set $\XX=\{1,\ldots,k\}$, $k\in\NN$. The elements of this set can be identified with the possible decisions of the players, respectively with the
(financial) positions in the financial instruments of a finite market. Our starting point of the so called closed-loop construction including an optimal control is the following forward Kolmogorov equation written in the weak form:
\begin{eqnarray}\label{eqn:kolmogorov}
 \frac{\prtl f_s}{\prtl s} - \left( \AF[s,\rho_s,u_s]f_s,\mu\right)
            &=& 0,\quad 0\le t < s\le T \\
    f(t,\xb)  &=&  \Phi(\xb),\quad \xb\in\hat\XX \nonumber
\end{eqnarray}
where $\mu$ is a finite measure in $\hat\XX$ and $f$ is an element of the dual space and a differentiable function in time,
the set of bounded continuous functions $\mathbf{C}([0,T]\times\hat\XX)$. The set is defined by $\hat\XX := \cup_{N=1}^\infty \XX^N$, where $\XX^N$ is the $N$-fold direct product of the set $\XX$, $\rhb$ is a function on $[0,T]$ taking values in the set of finite measures $\MM(\hat\XX)$. Finally the generator $\AF$ is of the form
\begin{eqnarray}\label{eqn:jumpgenerator}
 \lefteqn{\AF[s,\xb,\rho_s,u_s]f(s,\xb )
      =  \sum_{i=1}^{\norm{\xb}} \mathbf{A}^{i}[s,\xb,\rho_s,u_s]f(s,\xb )}\nonumber\\
     &=& \sum_{i=1}^{\norm{\xb}}\int_{\XX} \left(f_{i'}(s,y)-f_{i'}(s,x_i)\right)
                     u_s\nu(s,x_i,\rho_s,dy).
\end{eqnarray}
Here we introduce the notation $f_{i'}$ in order to describe that $\mathbf{A}^i$ acts on the component $x_i\in\XX$ only. In fact $f_{i'}(x_i) = f_{\xb_i'}(x_i)$ where $\xb_i'\in\XX^{\norm{\xb}-1}$ is derived by removing the variable corresponding to the $i^{th}$ agent from $\xb$. The length of the vectors describing the number of
players is denoted by $\norm{\xb}$.

\vskip 1.2cm

{\bf Hypothesis A}
\vskip 0.5cm
We assume $\nu_i(s, j, \rho, u)$ to be linear in the parameter $u=u_s$ and postulate $\nu_i(s, j, \rho)$ to be a bounded kernel in all parameters uniformly in $s$, $0\le t < s\le T$, and vanishing for $i=j$.
The choice of the space  $\XX$ means that the integral is a sum.\\
The parameters of the generator $\AF$ are subject to the assumptions that the control law $\ub\in\mathcal{U}$ satisfies $u_s\in U$ with bounded convex set $U$ having a smooth boundary, and that $\rhb$ is a Lipschitz continuous measure valued function on $[0,T]$ such that for all $s\in[t,T]$ we have $\rho_s\in \mathbf{P}_\delta(\hat\XX)$, the linear hull of Dirac probability measures. The natural domain of the operator $\mathfrak{D}(\AF[s,\rho_s,u_s])\subset \mathbf{C_0}(\hat{\XX})$ and $\mathbf{C_0}(\hat\XX)$ is the set of continuous functions vanishing at infinity on the discrete space $\hat\XX$ which will be restricted according to technical constraints.
\vskip 0.2cm
For the sake of completeness we add that $\mathbf{P}_\delta(\XX)\subset\MM(\XX)$, with $\MM(\XX)$ being the set of finite measures on $\XX$. An analogous
statement holds when replacing $\XX$ by $\hat\XX$.  We also introduce the set of continuous measure valued functions $\mathbf{C}([0,T],\,\MM(\XX))$, respectively,
$\mathbf{C}_\mu=\{\rhb\in\mathbf{C}([0,T],\,\MM(\XX))\mid \rho_0=\mu\}$ for later purposes. We mention that there exists an injection from $\hat\XX$ into $\NN$, which rises the question why we are using the notion $C(\hat\XX)$ of continuous functions. In fact, it seems advantageous at this stage to keep the analogy to jump processes on continuous spaces. Later we shall identify $C(\hat\XX)$ and $\RR^k$.
We also mention that $\MM(\XX)$ is isomorphic to $\RR^k$.
The sensitivity analysis is carried out on an open neighborhood $M\subset\RR^k$ of the origin.
\vskip 0.2cm
As mentioned above the construction involves a mean-field type limit consistent with a given optimal control problem. This is a particular example of measure valued limits from the theory of interacting particle systems. A key role within the toolbox of this theory plays the injection from the equivalence class $S\hat\XX$ of vectors $\xb\in\hat\XX$, which are identical up to a permutation of players, into the set of point measures
on $\XX$, defined by
\begin{displaymath}\label{x2Px}
  \xb = (\jmath_1,\ldots ,\jmath_N)\quad\longrightarrow\quad
  \fracN (\delta_{\jmath_1}+ \ldots + \delta_{\jmath_N})
  =: \fracN \delta_{\xb}\ .
\end{displaymath}
More precisely, for arbitrary $N\in\NN$ the mapping constitutes a bijection between $S\XX^N$ and the subset
$\Pmass^N_\delta(\XX)=\{\mu \in \MM(\XX),$ $\mu=\frac{1}{N} \sum_{k=1}^{N} \delta_k\}$ of $N$-point measures in $\XX$.
Implicitly we identify $\Pmass^N_\delta(\XX)$ in $\XX$
with the set of Dirac measure in $S\XX^N$.
For each $N\in\NN$ the space  $\mathbf{C_0}^{sym}(\XX^N)$
of $\RR$-valued continuous functions which are invariant
under component-wise permutations of their arguments is equivalent with the space of $\RR$-valued continuous functions $\mathbf{C}(S\XX^N)$.
Moreover, $\mathbf{C_0}^{sym}(\hat\XX)$ is a core of the operator $\AF[s,\rho_s,u_s]$. The restriction $\AF_N[s,\rho_s,u_s]$ of the operator $\AF[s,\rho_s,u_s]$ to
$\mathbf{C}(S\XX^N)$ generates a time inhomogeneous Markov process $X^N(s) = (X^N_1(s),\ldots ,X^N_N(s))$, $s\in[t,T]$, in $\XX^N$, see e.g. \cite{VC}, \cite{NJ}.

Since all agents are assumed to be subject to the same equation, the generator $\AF$ being of special form (\ref{eqn:jumpgenerator}), one investigates the dynamics for one representative of $N$ agents given
by the time inhomogeneous Markov process
$\mathbf{X}^N:= \mathbf{X}^N_i$, $1\le i\le N$, in $\XX$. As the number $N$ of agents tends to infinity the dynamics of the representative player depends on her own state and distribution only. Similar results from mathematical physics exist and physicists phrase this phenomenon: ''the individual dynamics in the mean field model separate as $N\rightarrow\infty$''.

By assumption the objective for each of the $N$ players is to find the value function
\begin{equation}\label{eq:valuefunction}
 V^N(t,x)= \sup_{\mathbf{u}}{I\!\!E}_{x}\left[\int_t^T
      J(s,X^N(s),\rho_s^N,u_s)\, ds + V^T(X^N(T),\rho_T)\right]
\end{equation}
on $[0,T]\times \XX$, i.e. to maximize her expected payoff over a suitable class of admissible control processes $\ub=\{u(s,X^N(s))\mid 0\le s\le T\}\in\mathcal{U}$. Here the cost function
$J: [0,T]\times\XX\times\mathbf{P}_\delta^N(\XX)\times U \rightarrow \RR$
and the terminal cost function $V^T:\XX\times\mathbf{P}_\delta^N(\XX)\rightarrow \RR$,
as well as the final time $T$ are given. With a particular choise we insure that the cost function is concave.
\vskip 0.2cm
An explicit expression for the value function can be derived by dynamic programming as solution of the HJB equation (\ref{25*}).
For admissible control processes the HJB equation is well posed and the resulting optimal feedback control function $\hat\ub^N$ is unique for
given start value $x\in\XX$ and given $\rhb$.
The so-called kinetic equation which leads to the nonlinear Markov process
in the sense of V. Kolokoltsov with control law  $\ub$
is derived by making an Ansatz motivated by the weak form of the one player evolution with an intrinsic choice of the parameter $\rhb$ in the generator:
\begin{equation}\label{16*}
 \frac{d}{ds}(g,\mu_s)
            = (\mathbf{A}[s,\rho_s,u] g,\mu_s)\vert_{\mub=\rhb}
 \footnote{The precise meaning of this equation in the setting of this paper is
 given in Theorem \ref{b}}
\end{equation}

for arbitrary $g\in \mathbf{C}(\XX)$ and arbitrary finite measures
$\mu_s\in\mathbb{M}(\XX)$ which are differentiable in $s\in[0,T]$.
To this end the corresponding differential equation for the adjoint
operator and the Koopman propagator to the nonlinear flow given by
the solution are investigated. The construction exhibits the order of convergence to be $\fracN$. The associated control problem reveals an
optimal feedback control $\ub$.
\vskip 0.2cm
Finally MFG consistency is said to hold if the fixed measure valued function $\rhb$ in the objective function can be replaced by the empirical measures
\begin{displaymath}
 \mu_s^N = \fracN (\delta_{X_1^N(s)}+ \ldots + \delta_{X_N^N(s)}), \qquad
 t\le s \le T,
\end{displaymath}
of the underlying process while well-posedness of the optimal control problem and uniqueness of the optimal control parameter are conserved - as a result of what could be called a closed loop construction.
This is realized by a fix point argument which establishes the $\fracN$-Nash equilibrium.
\vskip 0.1cm
We conclude the introduction with an overview of how the paper is organized. In Section~2 the dynamics of the game is introduced, in particular the Markovian propagator or time inhomogeneous semi group and the continuous in time Markov chain for one representative player. In Section 3 the limiting
dynamics is set up and the generator of the corresponding Koopman propagator is explicitly derived. The sensitivity analysis for the two associated control problems is discussed in Section 4. In the subsequent Section 5
the limit when the number of players tends to infinity is investigated.
Bounds for the approximation error are derived for the dynamics as well
as for the value functions. In the concluding section the $\fracN$-Nash equilibrium is established.

\section{Pure Jump Markov Processes}
In the entire section let us assume that the generator $\AF$ decomposes as given in
(\ref{eqn:jumpgenerator}). Hence we consider a single player $i$.
In order to simplify notations we even drop the index $i$, i.e. $\Ab:= \Ab^i$ whence $1\le i\le N$. At the same time the values of all agents different from $i$ are kept fix, i.e. $f(s,i):= f_{i'}(s,i)$, with the notation introduced
above.

Real valued functions on $\XX=\{1,\ldots, k\}$ can be represented as
$k$-vectors and consequenly the generator $\Ab$ as a $k\times k$-matrix. We assume that $\Ab$ is a time inhomogeneous $Q$-matrix on $\XX$,
i.e.
\begin{itemize}
 \item $-\infty < \nu_{i}(s,i,\rho_s)\le 0$ for all $i\in\XX$;
 \item $\nu_{j}(s,i,\rho_s)\ge 0$ for $i\ne j$, $i,j\in\XX$;
 \item $\sum_{j} \nu_{j}(s,i,\rho_s) = 0$ for all $i$.
\end{itemize}
Thus we have $\nu_{i}(s,\rho_s):=\nu_{i}(s,i,\rho_s)
                   =-\sum_{i\ne j} \nu_{j}(s,i,\rho_s)$
since the row sum vanishes.
We find using matrix form
\begin{eqnarray}\label{jumpmatrix}
\lefteqn{\Ab[s,\rho_s,u_s]\mathbf{f}(s)
   }\\
&=&\!\!
\begin{pmatrix}
  \nu_1(s,\rho_s)u_1 & \nu_2(s,1,\rho_s)u_2 & \cdots & \nu_k(s,1,\rho_s)u_k \\
	\nu_1(s,2,\rho_s)u_1 & \nu_2(s,\rho_s)u_2 & \cdots & \nu_k(s,2,\rho_s)u_k \\
  \vdots  & \vdots  & \ddots & \vdots  \\
  \nu_1(s,k,\rho_s)u_1 & \nu_2(s,k,\rho_s)u_2 & \cdots & \nu_k(s,\rho_s)u_k
 \end{pmatrix}\mathbf{f}(s)\nonumber
\end{eqnarray}
where $\mathbf{f}(s)=(f(s,1),\ldots,f(s,k))^t\in C([0,T])$ is $\RR^k$-valued.

\vskip 0.2cm
Since $\Ab$ is a finite dimensional matrix valued function of time we have the following:

\begin{equation}\label{eqn:Abound}
 \norm{\Ab f}\leq C \left\|f\right\|,
\end{equation}
$f\in\RR^k$ which means that the matrix valued function $\Ab$ constitutes a bounded linear operator.

\begin{proposition}
Let $M$ be a subset in the unit ball $B_1(0)\subset\RR^k$ and $U\subset\RR^k$ a convex bounded open control set. Assume that the matrix valued function $\Ab(s,\rho)$ in (\ref{jumpmatrix}) is continuous in $t > t_0$ for some $t_0\in\RR$, that it is of type  $C^q$ in the parameters $\rho\in \RR^k$.
Then so is the unique linear flow induced by $\nu$.
 \end{proposition}
The proof is a direct consequence of the results on linear ordinary differential equations in \cite{L} CH. XVIII section 4 and \cite{AH} CH 2. Since $\Ab x$ satisfies a linear growth
condition the unique global flow in the above theorem exists on the whole space.
\vskip 0.2cm

The solution to the Kolmogorov equation given by the matrix (\ref{jumpmatrix})
possesses the cocycle property,
which replaces the semi group property of autonomous systems see \cite{W}, \cite{A}. Intimately related to the cocycle property is the notion of a propagator to be found e.g. in physics publications or in works of Reed and Simon respectively V. Kolokoltsov.
A family of mappings $U^{t,s}$, $t\le s\le T$, in a set $S$ is called a
(forward) propagator (resp. backward propagator) in $S$ if $U^{t,t} = id_S$
and the following iteration equation holds:
\begin{displaymath}
  U^{t,r} = U^{t,s}U^{s,r}
\end{displaymath}
for $t\le s\le r$. Here $U^{t,s}U^{s,r}$ is to be interpreted as the iteration of mappings. For linear propagators or evolutions it means the application of linear operators, see \cite{DA} and \cite{VC}.

\begin{remark}
\label{11111}
The matrix valued functions $\Lambda(s,r,\cdot)$ constitute
bounded linear operators.

The family $\{\Lambda(t,s,\cdot) \mid 0\le t\le s\le T\}$
generated by the operator $\Ab$ defines a positive, strongly continuous linear propagator or evolution on the set
of Euclidean $k$-vectors which trivially coincides with the set of (continuous) real valued functions on the discrete set $\XX$.

\end{remark}
\vskip 0.1cm
\vskip 0.1cm
We now recall the connection between linear propagators or evolutions and
non-autonomous Markov processes which we intend to use for solving the control problem. We adopt the notation in \cite{DA} to the time dependent case.

Let $(\mathcal{E},\mathfrak{E})$ be a measurable space and $U^{t,r}$ an arbitrary linear propagator. Assume that $x\in\mathcal{E}, \, \mathrm{E}\in \mathfrak{E}$. We say that
$\{p(t,x,r,\mathrm{E}) :=(U^{t,r}\chi_\mathrm{E})$, where $\,0\le t\le r<\infty\}$ and $\chi_E$ is the indicator function of the set, is a normal transition family if
\vskip 0.1cm
\begin{enumerate}
 \item the maps $x\rightarrow p(t,x,r,\mathrm{E})$ are measurable for each $\mathrm{E}\in\mathfrak{E}$;
 \item the Chapman Kolmogorov equation holds;
 \item $p(t,x,r,\cdot)$ is a probability measure $\mathfrak{E}$.
\end{enumerate}

For the finite measurable space $(\XX, \mathfrak{P}(\XX))$ measurability in $x$ is trivially satisfied and the cocycle property together with the existence of a kernel reveal the Chapman Kolmogorov equation. The Markov property follows from the following proposition which is a straight forward adaption from \cite{N}.

\begin{proposition} A time inhomogeneous matrix
$Q(t),\, 0\le t\le s\le T$, on a finite set $I$ is a $Q$-matrix if and only if
$p(t,s)= T\!\exp \int_t^s Q(\tau)\, d\tau$ is a stochastic matrix for all $ 0\le t\le s\le T$.
\end{proposition}

We thus have:

\begin{lemma} The family $\Lambda$ generated by the matrix valued function $\Ab u$ is a normal transition family.
\end{lemma}

As done in \cite{NJ} the notion of a projective family for time inhomogeneous transition probabilities see \cite {DA} Theorem 3.1.7 holds for general measurable spaces $(E,\mathfrak{E})$ and trivially also to finite sets $\XX$, and arbitrary probability measures $\mu$ in $\XX$. The existence of a process is then guaranteed by the Kolmogorov existence theorem. We only state the existence of a process in the following

\begin{proposition} Given a normal transition family
 $\{p_{t,r}(x,K),\, 0\le t\le r<\infty\}$ and a fixed probability measure $\mu$ on the finite measurable space $(\XX,\mathfrak{P}(\XX))$, then there exists a probability
space $(\Omega,\, \mathcal{F},\, {I\!\!P}_{\mu})$, a filtration $(\mathcal{F}_t, t\ge 0)$ and a Markov process
$(X_t, t\ge 0)$ on that space such that:
\begin{equation*}\begin{array}{lc}
\PP{X(r)\in A \mid X(t) = x} = p_{t,r}(x,A) \mbox{ for each } 0\le t\le r,
   \ x\in\XX,\  A\in \mathfrak{P}(\XX)\ .\\
 X(0) \mbox{ has law } \mu\ .
\end{array}\end{equation*}
\end{proposition}
Since $\XX$ is compact the proof follows the line of arguments in Ethier Kurtz Theorem and  \cite{DA} Theorem 3.1.7. The result holds in general for Polish spaces.

In Section 5 we shall see that the solution of the kinetic equation (\ref{16}) is
the limit of the linear $N$-mean field evolutions as $N$ tends to infinity
which arises when restricting the generator $\AF$ in (\ref{eqn:kolmogorov}) to $C^{sym}(\XX^N)$ and replacing the parameter $\rho$ by the empirical distribution. The corresponding adjoint is denoted by $\AF^{*}_N$. In order to prove the mean field limit we need to unify spaces. This is possible since the factor spaces $S\hat\XX$ and the spaces of $N$-point measures $\Pmass^N_\delta(\XX)$ on the one hand as well as the corresponding larger spaces $\MM(\XX)$ and $\RR^k$ on the other hand can be identified.
We consider existence and uniqueness and the sensitivity analysis of the solutions of the Kolmogorov equation and the optimal control problems for the $N$-player on the larger space $\RR^k$.
\vskip 0.2cm

Consequently we replace or identify in a first step $\AF^{*}_N$ with the linear operator
\begin{equation}
 \mathfrak{\hat A}^N[t,\xb,\delta_{\xb},u] F(\delta_{\xb})
               :=\AF^{*}_N[t,\xb,\delta_{\xb},u]  f(\xb)
\end{equation}
on $C(\MM)$ for elements $\delta_{\xb}\in\Pmass^N_\delta(\XX)\subset\MM$.
The operator reads in more detail:
\begin{eqnarray*}
\lefteqn{{\hat\AF}^N [t,\delta_{\xb},u]   F(\delta_{\xb})	
   = \sum_{i=1}^N {\mathbf{A}^i}^{*}[t,\delta_{\xb},u]  F(\delta_{\xb})
}\\
 &=& \sum_{\ell =1}^k \sum_{\ell' =1}^k n_{\ell'}
      \nu_{\delta_{\ell'}}(t, \delta_{\ell}) u_{\ell'}
    \left[ F \left( \sum_{\ell =1}^k n_\ell  \delta_\ell
          + \delta_{\ell'} - \delta_\ell \right)
          -F\left( \sum_{\ell =1}^k n_\ell \delta_\ell \right)\right]
\end{eqnarray*}
where $\nu_{\delta_{\ell'}}(t, \delta_{\ell},u)$ is the transpose of
the matrix $\nu_{\delta_\ell}(t, \delta_{\ell'},u)$ and
$n_\ell$ describes how often the value $l$ appears and $n_{\ell'}$ is specified by the kernel $\nu$.

In a second step, specific to the case of a finite set $\XX$, we identify
$\{\delta_1,\ldots , \delta_k \}$  with the standard basis
$\{e_1,\ldots , e_k \}$ in $\RR^k$ to find that $\xb\in S\hat\XX$  with $\norm{\xb}= N$ corresponds to
$x^N = \sum_{\ell =1}^k n_\ell e_\ell\in\RR^k$ with
$\sum_{\ell= 1}^k n_\ell = N$ and:
\begin{eqnarray}\label{eqn:NMFgenerator}
\lefteqn{{\hat\AF}^N [t,\fracN x^N,u] F(\fracN x^N)
:= \sum_{i=1}^N {\mathbf{A}^i}^{*}[t,\fracN x^N,u] F(\fracN x^N)
}\\
 &=& \sum_{\ell =1}^k \sum_{\ell' =1}^k n_{\ell'} \nu_{\ell}(t, e_{\ell'}) u_\ell
    \left[F \left(\fracN \sum_{\ell =1}^k n_\ell e_\ell  + \fracN (e_{\ell'} -  e_\ell)\right)
          -F\left(\fracN \sum_{\ell =1}^k n_\ell  e_\ell \right)\right].\nonumber
\end{eqnarray}

\vskip 0.2cm

{\bf Hypothesis B}
\vskip 0.5cm
For the rest of the paper we complement the assumptions made on the domains of the variables and parameters of operator $\Ab$ by regularity conditions, namely:
We assume that $\nu$ is uniformly Lipschitz continuous in the parameter $x\in\RR^k$, and continuous in $t$.  Moreover, the partial derivatives  $\nabla_x \nu(t,i,x)$ are assumed to exist in $C_\infty(\mathbb R^k)$ as functions of $x$ and to be uniformly Lipschitz continuous, uniformly in the other variables. Finally let the cost function $J(s,i,x,u)$, $0\le t< s\le T$, to the control problems in Section 4 be quadratic concave in u and satisfy the same properties as $\nu$ regarding $s,i,x$ with $s$ replacing $t$.
\vskip 0.2cm
In the sequel we give an alternative representation of the linear operator $\hat\AF_t^N$. For practical reasons we introduce a scaling parameter
$h \in \mathbb R^+$. For differentiable functions $F$ on $\MM(\XX)$ a variational derivative $\frac{\delta F(Y)}{\delta Y(x)}$ of $F$ is the Gateaux derivative $D_{\delta_x}$ of $F$ in the direction of
$\delta_x$, $x\in\XX$. Since $\MM(\XX)$, and $\RR^k$ are isomorphic we are able to work with directional derivatives $\partial_{x}$ on $\RR^k$.
\vskip 0.2cm

\begin{proposition}\label{ANexpansion}
Assume that $\nu$ satisfies the Hypothesis A and that $M$ is a bounded open subset of $\RR^k$. Let $F\in C^2(\RR^k)$ then the operator $\mathfrak{\hat A}^N[t,hx^N,u]$, where $t\in[0,T]$, $x^N =  \sum_{i=1}^k n_i e_i\in\RR^k$ such that $\sum_{\ell=1}^k n_\ell = N$, $u\in U$, and
$h \in \mathbb R^+$, has the representation:
\begin{eqnarray*}
\lefteqn{ \mathfrak{\hat A}^N[t,h x^N,u] F(h x^N)
 = Nh\left( \Ab^*[t, h x^N,u]
          \partial_{h e_\ell}F(h x^N), e_\ell  \right)
        + h^2 \int_0^1\! ds(1-s)
         }\\
 &\times&\!\! \sum_{\ell,\ell',\ell''}\!\!\! \eta_\ell\nu_{\ell}(t,e_\ell)u_\ell
  \left(
  \partial_{h e_\ell}F( h x^N + h s(e_{\ell'} + e_{\ell''} - e_{\ell})),
 ((e_{\ell'}\otimes e_{\ell''}))
   \right)\ .
\end{eqnarray*}
Due to the fact that the agents are indistinguishable the sum with respect to $\ell$ becomes a factor $N.$
\end{proposition}
\begin{proof}
For $Y$ and $Y+\zeta$ such that the whole line
$\{Y+ \theta\zeta\mid 0\le \theta\le 1\}$ is in $M\subset\RR^k$ and
$F\in C^2(M)$ the Taylor theorem gives the following representation, see \cite{MHJ}, [\cite{K3} Cor 13 of Lemma 12.6.1]:

\begin{displaymath}
 F(Y+\zeta) - F(Y) = \left(\frac{\partial F(Y)}{\partial\zeta} ,\zeta\right)
+\int_0^1 ds(1-s)
\left(\frac{\partial^2 F(Y)}{\partial\zeta^2},\zeta\otimes\zeta\right)\ .
\end{displaymath}

Inserting the Taylor expansion of order 2 into (\ref{eqn:NMFgenerator})
under the integral for the choice
$Y=hx^N$ and $\zeta= h((e_\ell+e_{\ell'}) - e_\ell)$ finishes the proof.

We gladly anticipate that the first term coincides with the generator of
the Koopman propagator constructed in Proposition (3.4) for the choice
$h=\fracN$.
\end{proof}

Since the jump type operator $\mathfrak{\hat A}^N[t, x^N,u]$ is a linear combination of copies of the operator
$\Ab[t,h x^N]$ all properties previously investigated are conserved, in particular existence and sensitivity results, the fact that $\mathfrak{\hat A}^N[t, x^N,u]$ generates a strongly continuous contraction propagator $\psi^{t,s}_N$ and that there exists a corresponding Markov process. We also note that all statements in this section hold for $\Ab^*$ as well.

\section{Properties of the Non linear Evolution}

In the sequel we investigate the nonlinear Kinetic equation which
was motivated by the weak equation (\ref{eqn:kolmogorov}), namely
\begin{equation}\label{16**}
	\dot{\mu_s}=A^*[s,\mu_s,u_s]  \mu_s, \ \ \mu_t=\mu, \ \ s \in [t,T]\ ,
\end{equation}
$0 < t < T$. For the finite set $\XX$ the set of real valued functions on $\XX$, the set of bounded measurable and and the set of bounded continuous functions $C_\infty(\XX)$ coincide and are isomorphic to $\RR^k$. Consequently the dual space, the space of bounded measures $\MM(\XX)$, is isomorphic to $\RR^k$. When identifying $\MM(\XX)$ and $\RR^k$ the Kinetic equation, is the following nonlinear differential equation in $\RR^k$:

\begin{equation}\label{16}
	\dot{x_s}=A^*[s,x_s,u] x_s, \qquad x_t = x, \ \ s \in [t,T]\ ,
\end{equation}
$0 < t < T$.
Under the conditions of Hypothesis A the subsequent theorem gives the existence of a corresponding flow which is continuously differentiable
in all variables, parameters and initial conditions.

\begin{proposition}\label{prop:existence}
Let $M$ be a subset in $B_1(0)\subset\RR^k$ and $U\subset \RR^k$ a convex bounded
open control set. Assume that the matrix valued function $\Ab^*(s,x)$ in
(\ref{jumpmatrix}) is continuous in $t > t_0$ for some $t_0\in\RR$ and of type  $C^q$, for $q\ge 1$, in the variable $x\in M$. Then so is the unique nonlinear global flow $\alpha(t_0,t,x_0,u),\ t_0\le t$, arising from the solution of the Kinetic equation above. The unique global flow is defined on the whole space $\RR^k$.
\end{proposition}

Since $\nu$ is bounded and the set of admissible controls $U\subset\RR^k$ is bounded, the vector valued function $\nu(s,x)ux$ satisfies a uniform Lipschitz condition on $B_1(0)\subset\RR^k$. Hence $\Ab^{*}(s,x)$  is bounded in $M$. Moreover, the conditions of Theorems 2) and 7) as well Remarks 3) in \cite{L} CH. XVIII and Theorem 2.9 in \cite{AH} CH 2  apply which finishes the proof. Since the differential equation (\ref{16}) satisfies a linear growth condition The unique global flow extends to $\RR^k$.

This automatically implies that the solution $x_t$ of (\ref{16}) is
Lipschitz continuous in the initial conditions:

\begin{corollary}\label{b}
For all $x,\, y \in M \subset \RR^k $ the unique solution to equation (\ref{16}) given
by Proposition
\ref{prop:existence} is Lipschitz continuous in the initial data i.e:

\begin{equation}\label{17}
  \norm{\alpha(0,t,x,u) - \alpha(0,t,y,u)} \leq
          C(T)\left|x- y\right|
 \end{equation}
where $\norm{\cdot}$ is the Euclidean norm in $\RR^k$.	
\end{corollary}

We summarize our findings by concluding that the initial value problem (\ref{16})
is well-posed.

\begin{definition}
Let $\beta(t,s), \quad 0\leq t \leq s \leq T,$ be a nonsingular flow in a set
$K$ of a given Banach space, then
\begin{equation*}
(\Phi^{t,s} F)(b):=F(\beta(t,s,b)) \qquad \beta_t = b\in K
\end{equation*}
defines a linear operator on $C(K)$ which we call the Koopman propagator with
respect to $\beta(t,s)$.
\end{definition}
The notion coincides with the Koopman operator in \cite{LM}.
From the general theory, see \cite{LM}, it follows that the Koopman propagator has
the following properties:
\begin{itemize}
\item $\Phi^{t,s}$ is a linear propagator.
\item $\Phi^{t,s}$ is a contraction on $C(K)$, i.e.
   $\Norm{\Phi^{t,s} F} < \Norm{F}_{C(K)} \qquad\forall F \in C(K)$,
\end{itemize}
\vskip 0.05cm
where $\Norm{\cdot}$ denotes the norm in the Banach space.
Being a contraction the Koopman propagator is bounded.

For $M$ as in Proposition \ref{prop:existence} we introduce  the set
$C^1 (M)$ of functionals $F=F(x)$ such that the gradient $\nabla_{x} F$
is continuous.
This space becomes a Banach space when
equipped with the norm
\[
\Norm{F}_{C^1(M)}
     := \sup _{x \in K} \norm{(\nabla_{x} F)(x)}
.\]

\begin{proposition}\label{bbb}
Under the conditions given in Hypotheses A and B we have:
\begin{itemize}
\item[i)] The time inhomogeneous global
flow $\alpha$ corresponding to the solution $x_s$, $0\le t < s\le T$
of the kinetic equation defines the time inhomogeneous Koopman propagator:
\begin{equation}\label{171}
(\phi^{t,s} F)(x):=F(\alpha(t,s,x,u)) \qquad x\in\RR^k,\  0\leq t \leq s \leq T.
\end{equation}

\item[ii)] The generator of the Koopman propagator is defined by
\begin{equation}
\mathcal{A}[t,x,u ]F(x)
  =\sum_{i=1}^{k}\frac{\partial F}{\partial x_i} a_i^*[t,x,u]x ,
\end{equation}
where $a_i^*[t,x]$ corresponds to row $i$ of the matrix valued function
$\Ab^* [t, x]$.

\item[iii)] The Koopman propagator constitutes a strongly continuous family of\\
bounded linear operators on $C^1(M)$.
\end{itemize}
\end{proposition}

\begin{proof}
For the sake of a more comprehensive notation we drop the control parameter $u$.
First we shall prove that the global flow induced by the solution to the nonlinear kinetic
equation is nonsingular.

i) Let $x_t = x$. Since $\Ab^*[s,x,u]x$ is uniformly Lipschitz continuous hence satisfies a linear growth condition the unique global flow $\alpha$ given by Proposition 3.1 is defined on the whole space. Consequenly the flow constitutes a nonsingular transformation.
\vskip 0.1cm
ii) Under the assumptions at the beginning of Section 2 $(\nabla_{x} A^*) (x)$
exists and for every $x_0 \in \mathbb R^k$ the solution
$x_s = \alpha(0,s,x_0)$ exists for
all $s\in[0,T]$.  By inserting into the definition
we find
\begin{equation}
 \frac{(\phi^{t,s}F)(x)-F(x)}{s-t} = \frac{F(\alpha(t,s,x,u))-F(x)}{s-t}
  = \frac{F(x_s)-F(x)}{s-t}\ .
\end{equation}
where $x=x_t$ and $0\le t\le s\le T$. For $F \in C^1(M)$ with compact support the mean value theorem reveals
\begin{eqnarray}\label{251}
\frac{(\phi^{t,s}F)(x)-F(x)}{s-t}
 &=&\sum_{i=1}^k F_{x_i}(x_{\theta }) \dot{x_{\theta }}
     =  \sum_{i=1}^k F_{x_i}(x_{\theta })a_i^*[\theta,x,u] x_{\theta }\\
    &=& \sum_{i=1}^k F_{x_i}(\alpha(t,\theta ,x,u))
                  a_i^*[\theta,x,u] \alpha(t,\theta ,x, u)
\end{eqnarray}
where $t \leq \theta \leq s$. Since the derivatives $F_{x_i}$ have compact support by introducing the flow $\alpha(t,s,x_t)$ given by the solution $x_s$ into the kinetic equation (\ref{16}) with $x_t = x$ we obtain for $t\le \theta\le s$:
\[
\lim_{(s-t) \rightarrow 0} F_{x_i}(\alpha(t,\theta,x))\cdot a_i^*[t,x,u]\alpha(t,\theta,x)
= F_{x_i}(x)a_i^*[t,x,u] x
\]
uniformly $ \forall x\in M$ and $\forall F \in C^1(M)$  thus (\ref{251}) has a strong limit in $C^1(M)$
and the infinitesimal generator $\bf\mathcal{A}$ is giving by
\[
 {\bf\mathcal{A}}[t,x,u]F(x)
  =\sum_{i=1}^{k}\frac{\partial F}{\partial x_i} a_i^*[t, x,u]x\ .
\]
iii) To show the strong continuity we insert the definition of the Koopman propagator and exploit the properties of the flow given by Proposition
\ref{prop:existence}, i.e.we have:
\[
 \lim_{(t_0,s_0) \rightarrow (t,s)} \Norm{\phi^{t,s} F - \phi^{t_0,s_0} F}
 =0
\]
for every $F \in C^1(M)$.
The fact that the set of continuously differentiable functions with compact support form a dense subset of $C^1(M)$ concludes the proof.
\end{proof}

\section{Controlled Jump Markov Process}

In this subsection we shall describe the principle of dynamic programming and the corresponding HJB equation for the finite state Markov Jump Processes corresponding to the $N$-mean field dynamics and the Koopman propagator.
Two types of control problems with game theoretic applications will be covered: One simplified preliminary cost function $J$ which does not depend on an individual player but on the dynamics associated with the $N$-mean-field respectively Koopman propagator only. The other case where an individual player is introduced whose dynamics is given by the operator $\Ab$ and who is subject to the $N$-mean-field respectively mean field and both appear as measure valued parameters $\yb$ respectively $\yb^N$ in the cost function. In the first case the cost function $J:[0,T]\times\RR^k\times U  \rightarrow \RR$ defines
\begin{equation}\label{costfct1}
  \int_t^T J(s,X^N_s,u_s)\, ds + {V'}^T(X^N_T) ,
\end{equation}
where ${V'}^T(X^N_T)$ describes a terminal cost and $X^N$ is the Markov process generated by ${\hat\AF}^N$ in (2.4).
In the second case the optimal payoff for one player is represented by the value function $V:[0,T]\times\XX\times\RR^k  \rightarrow \RR$:
\begin{equation}\label{valuefctN}
V(t,j,y):= \sup _{\ub\in \mathcal{U}}
  \mathbb E_j \left[ \int_t^T J(s,X_s^{1},y_s,u_s)\, ds + V^T(X_T,y_T)\right]
\end{equation}
starting at time $t$ and position $j$. The process $X^{1}$ with generator $\Ab[s,j,y_s]$ is associated with the dynamics of the player. The measure valued parameter $\yb$ is replaced by the mean field respectively the $N$-mean field in the end. In the latter case we use the notation $V^N$ for the value function.
In order to guarantee a unique optimal control law in the set of Lipschitz continuous functions of the solutions, we confine to quadratic cost functions
\begin{equation}\label{costfct}
J(s,j,y,u) = \sum_{\ell = 1}^m J_{j,\ell} (s,y) u_\ell - \norm{u}^2,
\end{equation}
for $J_{j,\ell}\in \RR_{+}$ and $s\in [0,T], j\in \XX, y\in M, u\in U$ see \cite{X}.  We emphasize that the assumptions of Hypotheses A and B hold even for this section.

The methodology in a standard setting reveals the HJB equation i.e. the following system of ordinary differential equations
\begin{equation}\label{25*}
\frac{\partial V}{\partial t}+\max_{u}\left[\sum_{\ell = 1}^m J_{j,\ell}(t,y) u_\ell
                     - \norm{u}^2 + \Ab [t,y,u]V\right]=0\ .
\end{equation}
\begin{remark} Due to Hypothesis B, $\nu_j (t,i,y,u)\in C_0^1(\RR^k)$ in the variable $y$. The derivatives $(\frac{\partial}{\partial y} A[t,y])(x)$ are bounded and continuous.

\begin{itemize}
\item For any $t \in [0,T], \Ab[t,y]$ is $C^1_0-$differentiable with respect to the vector $y$,
also there exists a constant $c_1$ such that:
\begin{equation}
\label{11*}
 \sup_{(t,u)}\norm{\frac {\partial \Ab[t,y]}{\partial y}}_{\mathbb R^k}
\leq c_1\norm{y}
\end{equation}
\item $J(t,y)$ is $C^1-$differentiable with respect to the vector $y$, also there exists a constant $c_2$ such that:
\begin{equation}
\label{11**}
\sup_{(t,u)}\left|\frac{\partial J[t,y]}{\partial y}\right|_{\mathbb R^k} \leq c_2\left|y\right|
\end{equation}
\end{itemize}

\end{remark}

\begin{corollary} Assume that Hypotheses A and B hold.
Suppose that $\Ab[t,y]$ is as in (\ref{eqn:jumpgenerator}).
Then the solution of the ordinary differential equation (\ref{25*}) is well posed for all terminal data $V^T \in \mathbb R^k$ and the solution $V_t$ is of class $C^1$ in the parameters $y$ and $u$ for all $t \in [0,T]$.

Suppose that $J_{j,\ell} (s,\alpha,y)$ is $C^1$-differentiable with respect to the additional parameter $\alpha\in \RR$, then the solution $V_t$ is $C^1$-differentiable with respect to the parameters $y$, $u$, and $\alpha$.
\end{corollary}

The result is a direct consequence of Proposition \ref{prop:existence}.
\vskip 0.2cm
In order to cover the control problem (4.1) we need to generalize Corollary 4.2 in such a way that the parameter $y$ is replaced by a curve $\yb$. We study smooth dependence of the solution of the HJB equation (\ref{25*}) above when replacing the parameter $y\in\RR^k$ by a curve $\yb = y(t)$,
$t\in [0,T]$, in $\RR^k$, i.e.

\begin{equation}
\label{14}
 \frac{\partial V(t,\yb)}{\partial t}
               + \max_{u}\left[  \sum_{\ell = 1}^m J_{j,\ell} (t,\yb) u_\ell
               - \norm{u}^2 + \Ab [t,j,\yb,u] V(t,j)\right] = 0,
\end{equation}
respectively
\begin{equation}
\label{15}
H(t,j,V,\yb)=\max_{u}\left[\sum_{\ell = 1}^m J_{j,\ell} (t,\yb) u_\ell - \norm{u}^2
                         + \Ab [t,j,\yb,u] V(t,j)\right]\, .
\end{equation}
\vskip 0.2cm
Let us introduce a curve $\yb$ in the form of a piece of a straight line into the value function. For any $(t,j) \in [0,T]\times \XX$ and
$\yb^1, \yb^2 \in C([0,T],M)$ we define:
\begin{equation}
V(t,\alpha,j):= V(t,j,\yb^1
                +\alpha (\yb^2-\yb^1)) , \quad\alpha\in [0,1].
\end{equation}
Hence the smooth dependence on the solutions of the HJB equation on the functional vector valued parameter $\yb$ reduces to dependence on the real parameter $\alpha$. If the directional derivative $\partial_{\yb^2-\yb^1} V$ of $V(t,j,\yb)$ exists and is continuous we have:
\begin{eqnarray}\label{15**}
V(t,j,\yb^2)-V(t,j,\yb^1)=
\int_0^1{\partial_{\yb^2-\yb^1} V(t,j,\yb^1+\alpha((\yb^2-\yb^1)))d\alpha.} \nonumber
\end{eqnarray}

We adopt the assumptions we made when we studied smooth dependence on the real parameter $y$.

\begin{theorem}
Under the previous conditions and Hypotheses A and B we have that for any $\yb \in C([0,T],\RR^k)$ the solution of equation (\ref{14}) is Lipschitz continuous in $\yb$ uniformly, i.e. for $\yb^1, \yb^2 \in C([0,T],\RR^k)$, there exists a constant $K \geq 0$ such that

\begin{equation}
\sup \limits_{(t,j) \in [0,T]\times \XX}
\norm{{V(t,j,\yb^1)-V(t,j,\yb^2)}} \leq K \sup_{t \in [0,T]} \norm{\yb^1-\yb^2}.
\end{equation}
For measure valued functions $\yb \in \MM([0,T])$ the Euklidean $k$-norm
on the right hand side can be replaced by a weak*-norm, namely:
\begin{equation}\label{15*}
\sup_{(t,j) \in [0,T]\times \XX} \norm{{V(t,j,\yb^1)-V(t,j,\yb^2)}}
\leq K  \Norm{\yb^1-\yb^2}^{*}_2.
\end{equation}
where
$\Norm{\yb^1-\yb^2}^{*}_2
      := \sup_{f \in C^2([0,T],\RR^k)\vert\Norm{f}\le 1}(f,\yb^1-\yb^2)$.
\end{theorem}
The first result follows from Proposition 3.1 where the parameter $\yb$ is chosen from the Banach space of continuous functions with the supremum norm. The second result follows since the norm is weaker. We point out that $C^2$ is dense in $C$.

An analogous result holds for the value function $V$ associated with the Koopman dynamics.
Moreover, there exists a unique optimal Feedback control law $\hat u$ to the value function $V$ in (4.1).

\begin{proposition}\label{E*}
Under Hypotheses A and B, given a 
final payoff $V'^T$ the optimal control $\hat{u}$ defined by the cost function (4.1), is of feedback form $\hat{u}=\Gamma(t,\cdot)$ and is Lipschitz continuous i.e, for any
$\bm{\eta},\xb \in C_{y}\left([0,T],\RR^k\right)$
\begin{equation*}
\Gamma(t,\eta_t)-\Gamma(t,x_t)
\leq k_1 \sup_{s \in [0,T]}
\norm{\eta_s - x_s}_{\mathbb{R}^k}, \ \forall t \in [0,T],
\end{equation*}
\end{proposition}

\begin{proof}
A proof by Xu may be found in [30].
\end{proof}

The line of arguments and results presented above for the control problem with value function $V$ carries over to the modifications considered in this paper. We shall not repeat it.

\section{Convergence of N-particle Approximations}

In Physics and Biology scaling limits and analyzing scaling limits are well established techniques which allow to focus on particular aspects of the system under consideration. Scaling empirical measures by a small parameter $h$ in such a way that the measure $h(\delta_{x_1}+\ldots + \delta_{x_N})$ remains finite when the number $N$ of particles or species tends to infinity and the individual
contribution becomes negligible allows to treat the ensemble as continuously
distributed.

Scaling $k^{th}$-order interactions by $h^{k-1}$ reflects the idea that they are
more rare than $k-\ell$ order ones for $1\le \ell < k$ and makes them neither
negligible nor overwhelming. This scaling transforms an arbitrary generator
$\Lambda_k$ of a $k^{th}$-order interaction into
\begin{displaymath}
  \Lambda_k^hF(h \delta_{\xb})
  = h^{k-1} \sum_{I \subset \left\{1, \cdots,n\right\},\left|I\right|=k}
  \int_{\XX^k}{}\left[F(h \delta_{\xb} - h\delta_{\xb_I}+h\delta_y)
  -F(h\delta_{\xb})\right] \times P(x_{\xb_I};dy)
\end{displaymath}
with positive kernel $P(x_{\xb_I};dy)$.
The $N$-mean field limit is a law of large numbers for the first order interactions
given by the $N$-mean field evolutions. For the special case of pure jump type
$N$-mean field evolutions, cf. (\ref{eqn:jumpgenerator}), we prove weak convergence to
the solution of the kinetic equation (\ref{16}) by exploiting properties of the
corresponding propagators. The procedure consists of introducing the scale
$h= \frac {1}{N}$ and as explained
in Section 3 by unifying space, i.e. it is pursued by substituting
$f(\mathbf{x})$ by $F(\frac 1{\norm{\xb}} \delta_\mathbf{x})$ where $\norm{\xb}$ denotes the length of the vector.
In this section we adopt the representation on $\RR^k$ which was introduced in (2.4).The property exploited in the construction proving the $N$-mean field limit is:

\begin{proposition}
\label{aa}
Let $L^i:\mathbb{R}^k\rightarrow\mathbb{R}^k , i=1,2, t\geq 0,$ be two families of arbitrary bounded matrix valued functions which are continuous in time. Assume moreover, that $U^{t,r}_i$ are two linear propagators in $\mathbb{R}^k$ with $\left\|U^{t,r}_i\right\|\leq C_1, i=1,2$, such that for any $f \in \mathbb{R}^k$ the equation
\begin{equation}\label{eqn:1a}
\frac{d}{ds}U^{t,s}f= U^{t,s}L_s f, \quad \frac{d}{ds}U^{s,r} f=-L U^{s,r} f, \ \ t\leq s\leq r,
\end{equation}
holds in $\mathbb{R}^k$ for both pairs $\left(L^i,U_i\right)$.
Then we have
\vskip 0.1cm
\begin{equation*}\begin{array}{lc}
i)&U^{t,r}_2-U^{t,r}_1=\int^r_t{U^{t,s}_2(L_s^{2}-L_s^{1})}U^{s,r}_1ds\\
ii)&\left\|U^{t,r}_2-U^{t,r}_1\right\|_{B_1(0)\rightarrow \mathbb R^k} \leq C_1^2 (r-t) \sup \limits_{t\leq s\leq r}\left\|L_s^{2}-L_s^{1}\right\|_{B_1(0)\rightarrow \mathbb R^k}\ .
\end{array}\end{equation*}
\end{proposition}
The result is adopted from a well known result on bounded linear operators see e.g. \cite{EN}, and \cite{K3}. The representation is used to derive the subsequent properties.

The propagator $\Lambda(t,s,\cdot)$ generated by the operator $\Ab^*$ in
(\ref{jumpmatrix}) is bounded.

\vskip 0.2cm
As in Section 3 let $\psi_N^{t,s}, t \leq s,$ be the $N$-mean field propagator generated by (\ref{eqn:NMFgenerator}) and assume that $\phi^{t,s}$ is the Koopman propagator defined in (\ref{171}). Since the linear combination $\AF^N$ of copies of the bounded operators $\Ab^*$ in (\ref{eqn:jumpgenerator}) is also bounded, Remark \ref{11111} implies that $\psi_N^{t,s}$ is bounded.
In the first step let us compare the Koopman dynamics with the one of the $N$-mean field.
Exploiting Proposition \ref{aa}i) we derive an estimate for the deviation of the propagator $\psi^{t,s}_N$ from the Koopman propagator $\phi^{t,s}$. We first study the unrealistic case of a common initial condition.
\vskip 0.1cm
Let us fix the control parameter and set $x^N = \fracN \sum_{\ell=1}^k n_\ell e_\ell$ with $\sum_{\ell=1}^k n_\ell = N$.
By construction $\AF^N$ and $\psi^{t,s}_N$ satisfy equation (\ref{eqn:1a}) then Proposition \ref{aa} reveals:

\begin{equation}\label{29}
\left[(\psi_N^{t,s}-\phi^{t,s})F\right](x^N)\hspace{-2pt}
=\hspace{-2pt} \int_t^s \!\! {\left[\psi_N^{t,r}(\hat{\AF}^N[r,\alpha(t,r,x^N)]-\AK[r,\alpha(t,r,x^N)])
     \phi^{r,s}F\right](x^N) dr}
\end{equation}
for $F \in C^2 (M)$ and the flow $\alpha$ as in Proposition 3.1, independent of the control parameter $u\in U$.  We continue by estimating

\begin{eqnarray}\label{30}
\lefteqn{
 \sup_{y\in M}\norm{\left[(\psi_N^{t,s}-\phi^{t,s})F\right]\! (y\! )}\!
 \leq \!\! \int_t^s \! \Norm {\psi_N^{t,r}}
            \sup_{\Atop{y,y' \in M}{r\in[0,T]}}\!
                 \norm{(\hat{\AF}^N[r,y'] -\AK[r,y']) \phi^{r,s} F(y\! )}\! ds
}  \nonumber \\
 &\leq&\frac{(s-t)}{N}\Norm{\psi_N^{t,r}}
        \Norm{\hat{\AF}^N -\AK}
        \Norm{\phi^{r,s}} \Norm{F}_{C^2(M)}
 \leq \frac{C(T)}{N} \Norm{F}_{C^2(M)} \qquad\qquad
\end{eqnarray}
for $0 \leq t \le r \leq s \leq T,$ and
Proposition 2.6 was applied in the last step. The Sobolev type norm
$\Norm{F}_{C^2(M)}$ combines the supremum norm of $F$ and its second derivative.
The constant $C(T)$ summarizing the three operator norms and integration with respect to time.

This estimate will in a further step be applied to estimate the order of convergence in the mean field limit. The initial values are chosen to suit the operators and hence differ while $N$ changes. In fact, we shall assume that the initial conditions
\begin{equation}\label{300}
 x_0^N=\frac{1}{N}\sum_{\ell=1}^k \eta_{0,\ell} e_{\ell}\ \in\RR^k,
\end{equation}
with $\sum_{\ell=1}^k \eta_{0,\ell} = N$, of the Kolmogorov equation for the generators $\hat\AF^N$ converge in $\RR^k$,
as $N\rightarrow \infty$, to a vector $x_0 \in\RR^k$ in such a way that

\begin{equation}\label{3011}
\norm{x_0^N-x_0} \leq \frac{k_1}{N}\mbox{ with a constant } k_1 \geq 0\ .
\end{equation}

\begin{theorem}\label{FF}
Let the assumptions of Hypotheses A, B, and Proposition 5.1 be satisfied and let the initial conditions $ x_0^N \in\RR^k$ be subject to (\ref{3011}). Assume a fixed control parameter $\gamma\in U$. Then the following bounds hold:

\begin{itemize}
 \item [i)] For $t\in[0,T]$ with arbitrary $T\geq 0,$ we have
  \begin{equation*}
  \norm{(\psi^{0,t}_{N,\gamma} F)(x_0^N)-(\phi^{0,t}_\gamma F)(x_0)}
  \leq \frac{C(T)}{N} \left(T\Norm{F}_{C^2 (M)}+k_1\right)
  \end{equation*}
  with a constant $C(T)$ independent of $\gamma$;
 \item [ii)] For $J(t,x,\gamma)$ on $[0,T]\times\RR^k\times U$ there holds	
	\begin{equation*}
     \norm{{\int_t^T {\!\!J(s, x_{\gamma,s}^N, \gamma)ds}
      -\int_t^T {\!\!J(s,x_{\gamma,s},\gamma) ds}}}  \\
   \leq \frac{C(T)}{N} \left(T\Norm{J}^{*}_2 + k_1\right)
  \end{equation*}
\end{itemize}
where $x_{\gamma,t}^N $ is the law of the Markov process specified by the propagator $\psi_{N,\gamma}^{0,t}$ and $x_{\gamma,s}$ is the law of the
Markov process given by the Koopman propagator $\phi^{0,t}_\gamma$ or equivalently $x_{\gamma,s}=\alpha(t,s,x_0,\gamma)$. The norm $\Norm{J}^*_2$ was introduced in (\ref{15*}).
\end{theorem}
\begin{proof} i)
The chain of inequalities (5.3) holds uniformly for all $\gamma\in U$.
For $F \in {C^2 (M)}$ we have

\[
 \norm{(\psi_{N,\gamma}^{0,t}F)(x_0^N)\!-\!(\phi_\gamma^{0,t}F)(x_0)}
 \!\!\leq\!\! \norm{(\psi_{N,\gamma}^{0,t}F\!-\!\phi_\gamma^{0,t}F)(x_0^N)}
  + \norm{(\phi_\gamma^{0,t}F)(x_0^N)\!-\!(\phi_\gamma^{0,t}F)(x_0)}.
\]
Estimating the first term by (\ref{30}) and the second one by exploiting the Lipschitz continuity of the nonlinear flow guaranteed by Proposition
\ref{prop:existence} reveals the estimate.
ii) Let us represent the integral
   $\int_t^T {J(s, \fracN X_{\gamma,\gamma,s}^N,\gamma)ds}$ as the limit of Riemannian sums. Then the wanted inequality is obtained by applying Theorem 4.3 and part~i) term by term and passing to the limit.
   This finishes the proof.
\end{proof}
In the next section we shall assume that all agents are following a common strategy $\gamma(t,j)$ but one player, for instance the first one, who applies a different control $u_{1,t}=\tilde{\gamma}(t,j)$.
\nopagebreak

\section{Mean field limits as an $\epsilon$-Nash equilibrium}
A strategy portfolio $\Gamma$ in a game of $N$ agents with payoffs $V_i(\Gamma), i=1,...,N,$ is called an $\epsilon$-Nash equilibrium if, for each player $i$ and an acceptable individual strategy $u_i$
\[
 V_i(\Gamma) \geq V_i(\Gamma_{-i},u_i)- \epsilon,
\]
where $(\Gamma_{-i},u_i)$ denotes the profile obtained from $\Gamma$ by substituting the strategy of player $i$ with $u_i.$

Consequently we present two couples, consisting of a family of N-mean field games, where one player has a different preference than all others, and a game with the corresponding Koopman dynamics. For the sake of a shorter notation we use again $x^n:= \fracN \sum_{\ell}^k \eta_{\ell} e_{\ell}$.
In the first model the $N$-mean field acts as a single player and the differing preference $\tilde\gamma$ is part of it and in the second model this player is kept separate.
In the first setting the dynamics of $N$ interacting agents will be generated by the following operator

\begin{equation}\label{31}
\hat{\AF}^N [t, x^N,\gamma,\tilde\gamma] F( x^N)
  = \left[ \hat{\AF}^N [t, x^N,\gamma] + A^1[t, x^N,\tilde\gamma]
     - A^1[t,x^N,\gamma]\right] F( x^N)
\end{equation}
where $\hat{\AF}^N $ and $A^1=\Ab^*$ were defined in
(\ref{eqn:NMFgenerator}) respectively (\ref{eqn:jumpgenerator}), and $F \in C_\infty^2 (\RR^k)$.

\begin{remark}\label{def:psiN}
Let $\psi_{N,\gamma,\tilde{\gamma}}^{0,t}$ the $N$-mean field propagator on $C_\infty^1(\RR^k)$ generated by
$\hat{\AF}^N [t, x,\gamma,\tilde\gamma]$.
Since $\AF^N[t, x,\gamma,\tilde\gamma]$ is a linear combination of the linear operator $\Ab^*$, $\psi_{N,\gamma,\tilde\gamma}^{0,t}$ possesses the same properties as the propagator $\Lambda$ which is generated by $\Ab^*$, i.e. it is linear, for $T$ sufficiently small it is a contraction operator hence bounded, and it is Lipschitz continuous in the initial condition, i.e. it possesses the Feller property.
\end{remark}

\begin{theorem}\label{G}
Suppose Hypotheses A and B hold for the family of Markov jump type operators $\Ab^*[t,y,\gamma(t,.)]$ with a class of functions
$\gamma: \mathbb R^+ \times \mathbb\XX\rightarrow U,$ which are continuous in the first variable and Lipschitz continuous in the second one, let
$\psi_{N,\gamma,\tilde\gamma}^{0,t}$ be as above,
and let $\phi^{0,t}_\gamma$ the Koopman propagator (\ref{171}).
Then the following bounds hold:

\begin{itemize}
 \item [i)] Let $t \in [0,T]$ with any $T \geq 0$. For $F\in C^2(M)$:
  \begin{equation*}
	\norm{(\psi_{N,\gamma,\tilde\gamma}^{0,t}F)(x_0^N)
     -(\phi^{0,t}_\gamma F)(x_0)}
     \leq \frac{C(T)}{N} \left(T\Norm{F}_{C^2 (M)} + k_1\right);
  \end{equation*}
	with a constant $C(T)$ independent on $\tilde{\gamma}$.
 \item [ii)] Let $J(t,x,\gamma)$ be defined on $ [0,T]\times\RR^k\times U$. Then the following bounds exist:	
  \begin{equation*}
     \norm{{\int_t^T {\!\!J(s, x_{\gamma,\tilde\gamma,s}^N,\tilde\gamma)ds}
      -\int_t^T {\!\!J(s,x_{\gamma,s},\gamma)ds}}}  \\
   \leq \frac{C(T)}{N} \left(T\Norm{J}^{*}_2 + k_1\right)
  \end{equation*}

\end{itemize}
where $x_{\gamma,\tilde\gamma,t}^N$ is the law of the process specified by the propagator $\psi_{N,\gamma,\tilde\gamma}^{0,t}$ and $x_{\gamma, t}$ is 
the solution of the  kinetic equation (\ref{16}) with initial value $x_0$.
\end{theorem}
\begin{proof}
i) Having applied the representation (\ref{31}) the proof follows along the same lines as the one of Theorem (\ref{FF}).\\
ii) Let us represent the integral
   $\int_t^T {J(s, x_{\gamma,\tilde\gamma,s}^N,\tilde\gamma)ds}$ as the limit of Riemannian sums. Then the wanted inequality is obtained by applying part i) term by term and passing to the limit.
\end{proof}
Let us tend to the second setting. Let the first player have a differing preference and assume that in this case $J$ depends on the (unscaled) position of a tagged player, her differing strategy, and the empirical mean. So we have to look at the process of pairs $(X^{N,1}_{t},x_t^N)$, which refers to a chosen tagged agent and an overall mass. The generator of the pair $(X^{N,1}_{t},X_t^N)$ of processes are defined on the space $C_\infty^1 (\XX\times \RR^k)$ and take the form

\begin{equation}\label{ANtag}
%
 \hat{\AF}^N_{tag}[t,j_1, x^N,\gamma,\tilde\gamma]
            F(j_1, x^N)
 :=  \left(A^1[t,j_1, x^N,\bar\gamma] \right. \\
  \left. +\hat{\AF}^N
 [t, x^N,\gamma,\tilde\gamma]\right)  F(j_1, x^N),
\end{equation}
with $\hat{\AF}^N[t,j_1, x,\gamma,\tilde{\gamma}]$ as in (\ref{31}). The corresponding propagator will be denoted by,
$\xi_{N,\gamma,\tilde\gamma}^{0,t}$.

\begin{remark}\label{def:phiNtag}
Since $\hat{\AF}_{tag}^N$ is a finite sum of copies of the operator $\Ab^*$ the corresponding propagator $\xi^{0,t}_{N,\gamma,\tilde\gamma}$
possesses the same properties as the propagator $\Lambda$, generated by $\Ab^*$, and $\psi_{N,\gamma,\tilde\gamma}^{0,t}$ i.e.
it is linear, for $T$ sufficiently small it is a contraction operator hence
bounded, and it is Lipschitz continuous in the initial condition, i.e. it possesses the Feller property.
\end{remark}

\begin{remark}
Let $\phi_{\gamma}^{0,t}$ the propagator generated by the family
\begin{equation}\label{32}
\Ab^*[t,j_1,y,\tilde{\gamma}] + Id\,\AK[t,y,\gamma]
\end{equation}
on $C_\infty^1 (\XX\times \RR^k)$. Since the operator $\Ab^*$ is bounded and more regular in the parameters than $\mathcal A$, the propagator $\phi^{0,t}_{\gamma,\tilde\gamma}$ inherits the properties of the Koopman propagator $\phi^{0,t}$. Here we mention in particular that $\phi_{\gamma,\tilde\gamma}^{0,t}$ is a strongly continuous contraction.

By inserting (\ref{31}) into the definition and by applying Proposition \ref{ANexpansion} we find:
\begin{equation}\label{33}
 \hat{\AF}^N_{tag}[t,j_1, x^N,\gamma,\tilde{\gamma}]F(j_1, x^N )
  =  \left(A^1[t,j_1, x^N,\tilde{\gamma}]
     + Id\AK[t,y,\gamma]\right)F(j_1, x^N)+O(\fracN x).	
\end{equation}
\end{remark}

For the Kolmogorov equation corresponding to this generator we make the
assumptions on the initial conditions that
$x_{1,0}^N \in M$ converges, as $N \rightarrow \infty$ to a
point $x_{1,0} \in M$ such that

\begin{equation}\label{34}
 \norm{x_{1,0}^N-x_{1,0}} \leq \frac{k_2}{N}\hbox{ with a constant }k_2\ .
\end{equation}
\begin{theorem}\label{H}
Under the assumptions of the Theorem \ref{G}, let $\phi^{0,t}_{\gamma,\tilde\gamma}$,
$\xi^{0,t}_{N,\gamma,\tilde\gamma}$, and the cost function $J$ be as above. Then the following bounds exist for $t \in [0,T]$, $T \geq 0$:

\begin{itemize}
 \item[i)] For $F\in C^2 (\XX\times M)$ we have
  \begin{equation*}	\norm{(\xi^{0,t}_{N,\gamma,\tilde\gamma}F)(x_{1,0}^N,x_0^N)
     -(\phi^{0,t}_{\gamma,\tilde\gamma}F)(x_{1,0},x_0)}
     \leq \frac{C(T)}{N} \left(t\Norm{F}_{C^2(\XX\times M)}+k_1\right)
 \end{equation*}
	with a constant $C(T)$ not depending on $\gamma,\tilde{\gamma}$;
 \item[ii)] for $J(t,j,y,u)$ on $[0,T]\times\XX\times\RR^k\times U$:
  \begin{equation*}
  \begin{split}
	&\norm{\EE {\int_t^T J(s,X_{\tilde\gamma,s}^{N,1}, x_{\gamma,\tilde\gamma,t}^{N,1},
     \tilde{\gamma}(s))ds
     -\int_t^T J(s,X_{\gamma,s}^1,x_{\gamma,s},\gamma(s))ds} } \\
	&\leq \frac{C(T)}{N} \left((T+k_2)\Norm{J} +k_1\right)
	\end{split}
  \end{equation*}
\end{itemize}
where $\Norm{J} = \Norm{J}_{C(\mathcal{U})} + \Norm{J}^{*}_2$.
The pair
$\left(X_{\tilde\gamma,s}^{N,1}, \fracN X_{\gamma,\tilde\gamma,t}^N\right)$ is the Markov process specified by the propagator
$\xi_{N,\gamma,\tilde\gamma}^{0,t}$, and the process $X_{\gamma,s}^{1}$ is generated by $\Ab[t,x_{\gamma,s},\gamma]$.
Here $x_{\gamma,s}$
corresponds to the solution to the kinetic equation (\ref{16}) with initial condition $x_0$.
\end{theorem}

\begin{proof}
The basic idea is to insert definitions and to exploit the properties of the propagators $\Lambda^{t,s}$, $\psi_{N,\gamma,\tilde\gamma}^{t,s}$, and
$\xi_{N,\gamma,\tilde\gamma}^{t,s}$ summarized in Remarks 6.1 and 6.2 and the processes corresponding to the two first ones.

i) For $F \in C_{\infty}^2 (\XX\times M)$, we have
\begin{equation*}
\begin{split}
& \norm{\left(\xi_{N,\gamma,\tilde{\gamma}}^{0,t}  F \right)
 (x_{1,0}^N,x_0 ^N)
 - \left(\phi^{0,t}_{\gamma\tilde\gamma} F\right)(x_{1,0},x_0)}\\
  & \qquad \leq \norm{\left(\left(\xi_{N,\gamma,\tilde{\gamma}}^{0,t}
        - \phi^{0,t}_{\gamma,\tilde\gamma} \right) F\right)(x^N_{1,0}, x_0^N) }
       + \norm{  \left(\phi^{0,t}_{\gamma,\tilde\gamma} F\right)(x^N_{1,0},x_0 ^N)\right.\\
        & \qquad\quad  \left.
  - \left(\phi^{0,t}_{\gamma,\tilde\gamma} F\right)(x_{1,0},x_0^N)}
      + \norm{\left(\phi^{0,t}_{\gamma,\tilde\gamma} F\right)(x_{1,0},x_0 ^N)
        - \left(\phi^{0,t}_{\gamma,\tilde\gamma} F\right)(x_{1,0},x_0)  }.
\end{split}
\end{equation*}

We estimate the first term by (\ref{30}) using the operator norm and the second and the third ones by using the assumptions (\ref{34}) and (\ref{3011}) respectively. This finishes the proof of part i).
\vskip 0.1cm

ii) Let us represent both integrals
$\int _t^T {J(s,X_{\tilde\gamma,s}^{N,1}, x_{\gamma,\tilde\gamma,t}^N,
  \tilde{\gamma}(s ))ds}$
and \\
$\int _t^T J(s,X_{\gamma,s}^1, x_{\gamma,s},\gamma(s))ds$
as the limits of Riemannian sums. Then the result ii) is obtained by proceeding as in the proof of Theorem 5.2 and by applying Theorem 4.3 together with part~i) term-by-term and passing to the limit.
\end{proof}
The results of this and the previous two sections, and Theorem \ref{H} in particular are based on a fixed control parameter, depending on time however, and thus hold independently on the MFG methodology.

\begin{theorem}
Let $ \{A[t,j, y, u] \mid t \geq 0,j\in\XX, y \in M, u \in \mathcal{U}\}$ be the family of jump type operators given in (\ref{jumpmatrix}) and $\xb$ be the solution to equation (\ref{16}). Assume the following
\vskip 0.05cm
i) The kernel $\nu(t,j,y,u_t)$ satisfies the Hypotheses A and B;
\vskip 0.05cm
ii) The time-dependent Hamiltonian $H_t$ is of the form (\ref{15});
\vskip 0.05cm
iii) The terminal function $V^T $ is in $C^1_\infty(\XX\times\RR^k)$.
\vskip 0.05cm
iv) The initial conditions $x_0^N$ of an $N$ players game converge to $x_0$ in $\RR^k$ in a way that (\ref{3011}) is satisfied and (\ref{34}) holds.
\vskip 0.12cm
Then the strategy profile $u = \Gamma(t, x^N_{1,0},\alpha(0,t,x^N_0))$, defined via HJB (\ref{25*}) and (\ref{15}) is an $\epsilon$-Nash equilibrium in a $N$ players game, with
\[
 \epsilon = \frac{C(T)}{N}
 (\Norm{J}_{C(\mathcal{U})} + \Norm{J}^{*}_2
   + \Norm{V^T}_{C^2_\infty (\XX\times\RR^k) }+1).
\]
\end{theorem}
\begin{proof} Due to Assumption ii) the unique solution to the HJB equation admits a unique optimal control parameter given by (\ref{15}).
The optimal feedback control law $\Gamma=\Gamma(t,\alpha(0,t,x_{1,0},x_0))$ of the game with one player and the mean field with the Koopman dynamics and the cost function (\ref{costfct}) is applied to all players of the N-mean field possibly without one who applies $\tilde\gamma$. Then the operator difference in (5.2) at the point $(x^N_{1,0},x^N_0)$ reads
\begin{equation*}
 (\hat{\AF}^N[t,\alpha(t,r),\tilde\gamma,   \Gamma[r,\alpha(t,r,x_{1,0},x_0))]
  -\AK[t,\alpha(t,r),\Gamma[r,\alpha(t,r,x_{1,0},x_0)]
  )
     \phi^{r,s}
\end{equation*}
for the flow $\alpha(t,r)=\alpha(t,r,x^N_{1,0},x^N_0)$ as in Proposition 3.1 with an additional player.
Since the optimal feedback control $\Gamma$ depends on the flow $\alpha$
in a Lipschitz continuous way only, while the operator difference in (5.3)
is estimated by the second order term of a Taylor expansion, we uniformly
approximate by twice continuously differentiable functions using the Stone Weierstass Theorem. The assumptions are satisfied since the evolution is a contraction in a subset of a Euclidean unit ball.

Alternatively assume that the first player chooses a different strategy $\tilde\gamma$. The state dynamics of the first player, who is subject to an $N$-mean field, is described in terms of the process $X^{N,1}$. Let $(x^N_{1,0},x^N_0)$ be the initial condition. Then, we have
\begin{eqnarray*}
\lefteqn{\norm{ V^N (0,x^N_{1,0},x_{0}^N,\Gamma)
             -V^N (0,x^N_{1,0},x_{0}^N,\tilde\gamma)}
         }\\
&\leq&
 \norm{ \EXP_j \int_0^T \!\! J (s,X_{\Gamma,s}^{N,1},x^N_{\Gamma,s},\Gamma)ds
         - \EXP_j \int_0^T \!\!
   J (s,X^{N,1}_{\tilde\gamma,s},x^N_{\Gamma_{-1},s},\tilde\gamma )ds
    }\\
 && \qquad + \norm{\EE{V^T (X_{\Gamma,T}^{N,1})} -\EE{V^T (X_{\tilde\gamma,T}^{N,1})}}.
\end{eqnarray*}
For the proof the difference is rewritten such that the estimates in Theorems 6.2 and 6.5 can be applied taking into account the regularization mentioned above. We find
\[
\norm{V^1 (0,X_{\Gamma,0}^N) -\!V^1 (0,X_{\Gamma_{-1}, \tilde\gamma,0} ^N ) }
\!\! \leq  \! \frac{C(T)}{N}
     (\Norm{J}
      + \Norm{V^T} + \!1)
\]
where $\Norm{J}$ is the supremum norm in the components. It is clear that these estimates hold if we start the game at any time $t \in [0, T]$.   This completes the proof and the construction of the mean-field game in this paper.
\end{proof}

{\bf Acknowledgements}
The authors are expressing their deep gratitude to Sergio Albeverio for many years of fruitful collaboration, and for an outstanding course on the physical aspects of mean field theory. The PhD students from Linnaeus University are deeply indebted to Andreas Ioannidis for his beautiful, deep course on semi group theory.
Moreover, we would like to thank Diogo Gomes, Minyi Huang, Andrei Khrennikov and Torsten Linst\"{o}m for
stimulating discussions. The authors from Linneaus university gratefully achknowledge financial support by FTK, Linnaeus University.

\par\bigskip\noindent

\end{document}